\documentclass{amsart}

\usepackage{amsfonts,amsthm,amsmath,amssymb,graphicx}
\theoremstyle{plain}
\newtheorem{lem}{Lemma}[section]
\newtheorem{prop}[lem]{Proposition}
\newtheorem{thm}[lem]{Theorem}
\newtheorem{cor}[lem]{Corollary}
\theoremstyle{definition}
\newtheorem{defn}[lem]{Definition}

\theoremstyle{remark}

\numberwithin{equation}{section}

\DeclareMathOperator{\diam}{diam}

\DeclareMathOperator{\var}{var}

\newcommand{\R}{\mathbb R}

\newcommand{\N}{\mathbb N}
\newcommand{\Q}{\mathbb Q}

\renewcommand{\epsilon}{\varepsilon}

\usepackage[usenames]{color}

\numberwithin{equation}{section}

\begin{document}

\title[Multifractal analysis]{Multifractal analysis for expanding interval maps with infinitely many branches}

\author{Ai-Hua Fan}
\address{Ai-Hua Fan\\LAMFA UMR 6140, CNRS,
Universit\'e de Picardie Jules Verne, 33, Rue Saint Leu, 80039
Amiens Cedex 1, France}
\email{ai-hua.fan@u-picardie.fr}
\author{Thomas Jordan}
\address{Thomas Jordan\\School of Mathematics\\ The University of Bristol\\
University Walk\\Clifton\\ Bristol\\BS8 1TW\\UK}
\email{thomas.jordan@bristol.ac.uk}
\author{Lingmin Liao}
\address{Lingmin Liao\\LAMA UMR 8050, CNRS,
Universit\'e Paris-Est Cr\'eteil, 61 Avenue du
G\'en\'eral de Gaulle, 94010 Cr\'eteil Cedex, France}
\email{lingmin.liao@u-pec.fr}
\author{Micha\l\ Rams}
\address{Micha\l\ Rams\\Institute of Mathematics\\ Polish Academy of Sciences\\ ul.
\'Sniadeckich 8, 00-956 Warszawa\\ Poland }
\email{M.Rams@impan.gov.pl}
\thanks{M.R. was partially supported by the MNiSWN grant N201 607640 (Poland).}
\thanks{The work was started during a conference in Warsaw in October 2010 on Fractals in Deterministic and Random Dynamics. This workshop was funded by the EU network CODY.
 %and we would like to thank them for their support. 
 The research was continued in a workshop in Warwick in April 2011 on Dimension Theory and Dynamical Systems. This workshop was funded by the EPSRC.
 We would like to thank both workshops for their support.}


\date{}


\begin{abstract}
In this paper we investigate multifractal decompositions based on values of Birkhoff averages of  functions from a class of symbolically continuous functions. This will be done for an expanding interval map with infinitely many branches and is a generalisation of previous work for expanding maps with finitely many branches. We show that there are substantial differences between this case and the setting where the expanding map has only finitely many branches.
\end{abstract}

\maketitle
%\def\thefootnote{}
%\footnote{2010 {\it Mathematics Subject Classification}: Primary
%28A80 Secondary 37E05, 28A78}

%\def\thefootnote{\arabic{footnote}}

\section{Setting}

Let $(X,d)$ be a metric space and $T: X\rightarrow X$ be a piecewise continuous transformation. Let $\phi : X \rightarrow \mathbb{R}$  be a real-valued function (called a potential).  The Birkhoff average of $\phi$ is defined by
\[
      A_n\phi(x):= \frac{1}{n}\sum_{j=0}^{n-1}\phi(T^jx).
\]
With respect to an ergodic measure, for a measurable potential $\phi$, the Birkhoff averages $A_n\phi(x)$ almost surely converge to the integral of $\phi$. However, since for an expanding map there is a large family of ergodic measures, the Birkhoff averages can take a wide variety of values. From the point of view of multifractal analysis, one considers the size (Hausdorff dimension) of the level sets of the limit of the Birkhoff averages. That is, for a given level $\alpha\in \mathbb{R}$, the Hausdorff dimension of the set
\[
\Big\{x\in X: \lim_{n\to \infty}A_n\phi (x) = \alpha \Big\}.
\]

  	There has been a substantial amount of works on this multifractal analysis, especially for expanding interval maps with finitely many branches. The first example we know where a problem of this type was studied is the work of Besicovitch in \cite{bes} on the Hausdorff dimension of sets determined by the frequency of the digits in dyadic expansions. This can be viewed as a multifractal analysis of the Birkhoff averages of the indicator functions for the doubling map. This work was subsequently extended by Eggleston \cite{E} and many others \cite{BSc, Caj, Du, Oli98, Oli00, Ol02, Ol03, OW03, PS07, Vol58}.  For a continuous potential, the case of mixing subshift of finite type is studied in several papers including \cite{BS, BSS, BSS2, FF, FFW, FLW, O, Oli1,OW, PW,T}. In \cite{FLW} Feng, Lau and Wu proved a conditional variational principle for continuous potentials in the setting of general conformal expanding maps and in \cite{BS} Barreira and Saussol showed that this conditional variational principle varies analytically for H\"{o}lder potentials. In \cite{TV} Takens and Verbitzkiy considered systems with specification property and calculated the topological entropy of the level sets. In \cite{H}, Hofbauer studied the entropy of the level set of Birkhoff averages for piecewise monotone interval maps. It is also possible to study a countable family of piecewise continuous potentials.  This case was investigated by Olsen \cite{O}, Olsen and Winter \cite{OW}  for subshifts of finite type and conformal iterated function systems and by Fan, Liao and Peyri\`ere \cite{FLP}, in terms of topological entropy, for systems satisfying the specification property.

In particular in  \cite{O},
% {\bf (it seems that the author only considered the case of symbolic systems ? Is the $C^{1+\eta}$ included?) (Basically yes since conformal iterated function systems are considered)}
the following situation is considered. Let $T:[0,1]\rightarrow [0,1]$ be a $C^{1}$ expanding map and for $i\in\N$ let $\phi_i:[0,1]\rightarrow\R$ be continuous functions. For a vector $\underline{\alpha}\in\R^{\N}$  let
$$X_{\underline{\alpha}}:=\left\{x\in [0,1]:\lim_{n\rightarrow\infty}\frac{1}{n}\sum_{j=0}^{n-1}\phi_i(T^jx)=\alpha_i\text{ for all }i\in\N\right\}.$$
It is shown that if $X_{\underline{\alpha}}\neq\emptyset$ there exists a $T$-invariant measure $\mu$ such that $\int\phi_i\text{d}\mu=\alpha_i$ for all $i\in\N$ and
$$\dim X_{\underline{\alpha}}=\sup\left\{\frac{h(\mu)}{\lambda(\mu)}:\int\phi_i\text{d}\mu=\alpha_i\text{ for all }i\in\N\right\}.$$
%Here the supremum is taken over all $T$-invariant probability measures,
Here $\dim$ stands for the Hausdorff dimension, $h(\mu)$ denotes the measure theoretic entropy of $\mu$ and $\lambda(\mu)=\int\log(|T'(x)|)\text{d}\mu$ is the Lyapunov exponent of $\mu$.

The aim of this paper is to look at expanding maps $T$ on a non-compact space where $T$ has a countable number of inverse branches. While much of the same theory still holds there are also substantial differences.

In the setting of expanding maps with a countable number of branches, there have been several papers looking at multifractal analysis. Most of these papers concentrate on the local dimension of Gibbs' measures or specific examples of continuous potentials, for example
$\log |T'|$ concerning the Lyapunov exponent. Of particular relevance to our work are the papers \cite{FLM} and \cite{FLMW} which consider the frequency of digits for certain maps with a countable number of branches. This can be viewed as an example of multifractal analysis for Birkhoff averages of a specific family of potentials. A notable feature of these papers is that frequencies of digits which sum to less than $1$ still yield positive dimension. However such sets cannot be related to an invariant measure. There is also a preprint \cite{IJ}, which considers the case of one piecewise continuous potential with certain properties.   Our main aim is to generalize these results to more general families of potentials and more general countable expanding maps with a countable number of branches. Related questions are also studied in certain non-conformal settings \cite{R, KR}.

Let $\{I_i\}_{i=1}^{\infty}$ be a countable collection of disjoint subintervals of $[0,1]$. %such that  $I_i\subset [0,1]$ for all $i$ and for
     %$I_i\cap I_j=\emptyset$ all $i,j\in\N$ with $i\not=j$.
Let
$T_i:\overline{I_i}\to [0,1]$ be a bijective $C^{1}$ map such that
$|T_i'(x)|\geq\xi>1$. By this we will mean that $T_i$ can be extended to a $C^1$ diffeomorphism from an open neighbourhood of $I_i$ to an open neighbourhood of $[0,1]$ which maps $\overline{I_i}$ to $[0,1]$. We define the map $T:\cup \overline{I_i}\rightarrow
[0,1]$ as follows. If $x$ is not a common end point of two intervals, define
$$T(x)=T_i(x)\ \text{ if }x\in \overline{I_i}.$$
Otherwise we simply set $ T(x)=T_l(x)$ where $l=\min\{j:x\in I_j\}$.

Consider the full shift $(\Sigma, \sigma)$ with $\Sigma=\N^{\N}$ and the natural
projection $\Pi:\Sigma\rightarrow [0,1]$ defined by
$$\Pi(\underline{i})=\lim_{n\rightarrow\infty}T_{i_1}^{-1}\circ\cdots\circ T_{i_n}^{-1}([0,1]).$$
Let $$
\Lambda=\Pi(\Sigma). %\cap_{n=0}^{\infty}T^{-n}([0,1]).
$$
Then $(\Lambda, T)$ defines a dynamical system. We remark that the space $\Lambda$ is not necessarily compact and it could also be a Cantor type set. We will denote
$$E:=\{x\in\Lambda:\#\Pi^{-1}(x)\geq 2\}$$
and note that this set is at most countable  and so for any set $\Omega\subset\Lambda$ we have that $\dim\Omega=\dim\Omega\backslash E$. To avoid confusion with the notion of the derivative of $T$ we will adopt the convention that for $x\in\Lambda$ $T'(x)=T_l'(x)$ where $l=\min\{j:x\in I_j\}$.
We will also assume that
     the variations of $\log |T'|$ converge uniformly to $0$ (defined precisely in Section 2, see Definition \ref{unifvar}).

    %{\color{red} ($T$ is not well defined on the union of $I_i$, the problem being on the common end point
%of two $I_i$'s if there are any. That is the case of the continued fractions. Consequently $\Lambda$
%is not well defined because $T^{-1}$ is not well defined.)}

%We will  relate the dynamical system $(\Lambda, T)$ to the
%{\color{red} (why $\Pi$ maps into $\Lambda$?   How to understand $T_i^{-1}$ ? i.e. $T_i$ is defined on $\overline{I}_i$ or
%$I_i$ ? It seems that $T_i$ is defined on $\overline{I}_i$, otherwise the limit defining $\Pi(i)$ may be empty. But the limit
%point may not in $\Lambda$. I am confused.)}

Let $\mathcal{M}(T)$ be the set of $T$-invariant probability measures on $\Lambda$ and note that they must assign $0$ measure to $E$.  Thus $\Pi$ gives a bijection between the set of shift invariant probability measures and $T$-invariant probability measures. To avoid complications when we refer to weak* limits of a sequence of measures we will always mean weak* limits of the measures in the symbolic space.

%We will consider the case where each $\phi_i$ is a bounded
%function, that is for all $i$ there exists $K_i>0$ such that for
%all $x\in\Lambda$ $|\Phi_i(x)|\leq K_i$.

%Using this map we can relate We will denote by $\Sigma_n=\N^n$ the set
%of sequences of length $n$.

     Given a sequence of functions $\phi_i:\Lambda\rightarrow\R$  ($i\in\N$), which satisfy that the variations tend uniformly to $0$ (again see Definition \ref{unifvar}),
we will denote the Birkhoff averages
$$A_n\phi_i(x)=\frac{1}{n}\sum_{j=0}^{n-1}\phi_i(T^jx).
$$ We wish to
study the possible limit points in $\R^\N$ of the Birkhoff average sequences
$\{A_n\phi_i(x)\}_{n\in\N}$
by investigating the sets of the form
$$\Lambda_{\underline\gamma}=\{x\in\Lambda:\lim_{n\rightarrow\infty}A_n\phi_i(x)=\gamma_i\text{ for all }i\in\N\}
,\quad \underline{\gamma} \in \mathbb{R}^\mathbb{N}.$$

The following sets will describe the possible limits of the Birkhoff averages. Let $$
Z_0=\left\{\underline{\gamma}\in\R^\N: \exists \mu \in \mathcal{M}(T),\forall i\in\N,
     \int\phi_i\text{d}\mu=\gamma_i
\right\}.
     $$
     We will denote by $Z$
the closure of $Z_0$ in the pointwise limit topology. That is to say,
$\underline{\gamma}\in Z$ means that for any $\epsilon>0$ and any $k\in\N$
there exists a $T$-invariant probability measure $\mu$ such that
$$\forall 1\leq i\leq k, \ \  \left|\int\phi_i\text{d}\mu-\gamma_i\right|\leq\epsilon. $$

For a $T$-invariant probability measure $\mu$ let $h(\mu)$ and $\lambda(\mu)$ denote the measure theoretic entropy and the Lyapunov exponent of $\mu$ respectively. See Section 2 for formal definitions.

Our aim is to find the Hausdorff dimension of $\Lambda_{\underline{\gamma}}$ and consider how the dimension varies with $\underline{\gamma}$.
The known results for dynamical systems of finite branches suggest three  natural candidates in the infinite case.
Given $\underline{\gamma}\in Z$, let
\[
\alpha_1(\underline{\gamma})=\lim_{\varepsilon\to 0}
\lim_{k\to\infty}
\sup_{\mu\in\mathcal{M}(T)}\left\{\frac{h(\mu)}{\lambda(\mu)}:\left|\int\phi_i
\text{d}\mu-\gamma_i\right| < \varepsilon\ \forall i\leq k,
\ h(\mu)<\infty\right\}.
\]
Let $\alpha_2$ be a similar function, the difference being that the supremum is taken over ergodic measures ($\mathcal{M}_{\mathcal{E}}(T)$):
\[
\alpha_2(\underline{\gamma})=\lim_{\varepsilon\to 0} \lim_{k\to\infty}
\sup_{\mu\in\mathcal{M}_{\mathcal{E}}(T)}\left\{\frac{h(\mu)}{\lambda(\mu)}:\left|\int\phi_i
\text{d}\mu-\gamma_i\right| < \varepsilon\ \forall i\leq k, \
h(\mu)<\infty\right\}.
\]
Finally, for $\underline{\gamma}\in Z_0$ we will define
\[
\alpha_3(\underline{\gamma})=\sup_{\mu\in\mathcal{M}(T)}\left\{\frac{h(\mu)}{\lambda(\mu)}:\int\phi_i \text{d}\mu=\gamma_i \ \forall i\in\N, \ h(\mu)<\infty\right\}.
\]

We can now state our main theorems.
\begin{thm}\label{main0}
For $\underline{\gamma}\notin Z$, we have $\Lambda_{\underline\gamma}=\emptyset$.
For all $\underline{\gamma}\in Z$, we have
\[
\dim \Lambda_{\underline\gamma}= \alpha_1(\underline{\gamma})=\alpha_2(\underline{\gamma}).
\]
\end{thm}
We would like to state the spectrum using the function $\alpha_3$,
too (hence, without the awkward limits in $k$ and $\varepsilon$).
However, as shown in \cite{FLM} and \cite{FLMW}, the spectrum
is not necessarily equal to $\alpha_3(\underline{\gamma})$. One
particular problem is that there might be points in $Z\setminus
Z_0$ that are limits of the Birkhoff averages of $\phi_i$ for some $x\in\Lambda$
(while, not belonging to $Z_0$, they are not averages of
potentials $\phi_i$ for any invariant measure). Another problem is
that even for points in $Z_0$ the spectrum needs not to be the
supremum of $h/\lambda$ over invariant measures with given
averages of $\phi_i$.

We are only able to present the ``exact"
type statement for bounded potentials, and the proof involves more steps
than  for the ``approximate" type statements of Theorem \ref{main0}. We
also need to introduce the quantity,
$$s_{\infty}=\inf\Big\{s\ge 0:\sum_{i\in\N}\diam(I_i)^s<\infty\Big\}.$$
Observe that $0\le s_\infty \le 1$. The exponent $s_\infty$ will play an important role.

\begin{thm}\label{main}
If the potentials $\phi_i$ are all bounded then for all
$\underline{\gamma}\in Z_0$ we have
$$\dim \Lambda_{\underline\gamma}=\max\left\{s_{\infty}, \ \alpha_3(\underline{\gamma})\right\},$$
while for all $\underline{\gamma}\in Z\setminus Z_0$ we have
\[
\dim \Lambda_{\underline\gamma}=s_\infty.
\]
\end{thm}
%In particular note that for any bounded function $\phi_i$ we have that $s^*(\phi_i)=s_{\infty}$. Thus Theorem \ref{main} applies in the case when all the functions $\phi_i$ are bounded functions.
\medskip
The rest of the paper is structured in the following way. In Section 2 we give some results based on the distortion of the functions and the topological pressure. Next we use Section 3 to introduce the main tools we will use to prove Theorems \ref{main0}   and \ref{main}. Section 4 gives the proof of Theorem \ref{main0} and the proof of Theorem \ref{main} is given in Sections 5 and 6. Finally in Section 7 we give some examples of our results, including frequency of digits, harmonic averages for continued fractions and multifractal spectra with flat regions.

At the end of this section, we would like to give a list of the notation which will be used in  this paper.
\begin{itemize}
\item $\Sigma=\N^{\N}$ : the full shift with the shift transformation $\sigma$.
\item $\Sigma_q=\{1, \dots, q\}^{\N}$ : the symbolic space of $q$ symbols.
\item $[\omega_1,\cdots,\omega_n]$ : $n$th level cylinder set in $\Sigma$.
\item $C_n(\omega)=C_n(x)=C_n(\omega_1\cdots\omega_n)$ with $x=\Pi \omega, \omega \in [\omega_1,\cdots,\omega_n]$ :  $n$th level basic interval in $\Lambda$.
\item  $\phi,  \phi_i$ : functions on $\Lambda$; $f=\phi\circ\Pi, f_i=\phi_i\circ\Pi$ : the corresponding functions on $\Sigma$.
\item $A_n\phi_i(x)=\frac{1}{n}\sum_{j=0}^{n-1}\phi_i(T^j x)$ : Birkhoff averages of $\phi_i$.
\item $\mu, \mu_j$ : measures on $\Lambda$; $\nu,\nu_j, \eta, \eta_j$ : measures on $\Sigma$.
\item $\lambda_i$ : the maximal contraction ratio of map $T_i$.
\item $\psi_k(x)=\frac{1}{k}\sup_{y\in C_k(x)}\log |(T^k)'(y)|$.
\item $\xi_k(\mu)=\int\psi_k\text{d}\mu$.
\item $\dim A$ : Hausdorff dimension of a set $A$.
\item $h(\mu)$ : entropy of $\mu$.
\item $\lambda(\mu)$ : Lyapunov exponent of $\mu$.
\item For $(\omega_1,\ldots,\omega_n)\in\N^n$, $\overline{(\omega_1,\ldots,\omega_n)}$ denotes the periodic point $\tau\in\Sigma$ such that for any $a\in\N$ and $1\leq b\leq n$ $\tau_{an+b}=\omega_b$.
\end{itemize}

\section{Topological pressure and Distortion}

We first introduce several useful quantities (including  entropy,
Lyapunov exponent, pressure) and a variational condition on
potentials which ensures a distortion result.

We start by defining cylinders and basic intervals in our setting.
%\begin{defn}
Let $\omega\in\Sigma$. Denote by $[\omega_1,\cdots,\omega_n]$ the $n$th level cylinder. The $n$th level basic interval determined by $\omega$ is
$$C_n(\omega)=C_n(\omega_1,\ldots,\omega_n)=T_{\omega_1}^{-1}\circ\cdots\circ T_{\omega_n}^{-1}([0,1])\backslash E.$$
Sometimes, we also write this basic interval by $C_n(x)$ with $x=\Pi \omega$.
%\end{defn}
%{\bf ($T_i$ is defined on $\overline{I}_i$? Then the cylinder is a closed interval
%and two such cylinders are not necessarily disjoint. They are only touching a tend points so it is never a problem, the associated ifs satisfies the OSC with $(0,1)$ as the open set. )}

Two key concepts for this paper will be the measure theoretic entropy  and the Lyapunov exponent of an invariant measure.
For a $T$-invariant probability measure $\mu$ we define its entropy (\cite{MUbook}, pages 292-293) by
$$h(\mu)=\lim_{n\rightarrow\infty}\frac{1}{n}\sum_{(\omega_1\ldots,\omega_n)\in\N^n}\mu(C_n(\omega_1,\ldots,\omega_n)) \cdot\log\mu(C_n(\omega_1,\ldots,\omega_n))$$
and its Lyapunov exponent by
$$\lambda(\mu)=\int\log |T'(x)|\text{d}\mu(x).$$
It is well known that $h(\mu)\leq\lambda(\mu)$. However it is possible that they could both be infinite.

We now consider the regularity conditions we will need our potential functions $\phi_i$ to satisfy.
For $\phi:\Lambda\rightarrow\R$ define its $n$th variation by
$$\var_n(\phi)=\sup\{|\phi(x)-\phi(y)|: \ x,y\in C_n(\omega), \omega= (\omega_1,\ldots,\omega_n)\in\mathbb{N}^n\}.$$
It is clear that $\var_n(\phi)$ decreases as $n$ tends to $\infty$ and that $\lim_n\var_n(\phi)=0$
means  $f:=\phi\circ \Pi$ is uniformly continuous  on $\Sigma$ when $\Sigma$ is equipped with the usual metric.

\begin{defn}\label{unifvar}
Let $\phi:\Lambda\rightarrow\R$. We say that $\phi$ has
variations uniformly converging to $0$ if $\var_1(\phi)<\infty$
and $\lim_{n\rightarrow \infty}\var_n(\phi)=0$.
\end{defn}

Given a basic interval $C_n(\omega_1,\ldots,\omega_n)$ we define
\[
M^*\phi(\omega_1,\ldots,\omega_n) = \sup_{x\in
C_n(\omega_1,\ldots,\omega_n)} A_n\phi(x)
\]
\[
M_*\phi(\omega_1,\ldots,\omega_n) = \inf_{x\in
C_n(\omega_1,\ldots,\omega_n)} A_n\phi(x).
\]
%Recall that $\psi = \phi\circ \Pi$.

\begin{lem}\label{ancyl}
Let $\phi:\Lambda\rightarrow\R$ have variations uniformly tending to $0$.
Then
\[
\lim_{n\to\infty} \sup_{(\omega_1\ldots,\omega_n)\in\N^n}
M^*\phi(\omega_1,\ldots,\omega_n)-M_*\phi(\omega_1,\ldots,\omega_n)=0.
\]
\end{lem}
\begin{proof}
%Let $\psi=\phi\circ\Pi$.
The result immediately follows from the following estimation: for fixed $n\in\N$ we have
$$|M^*\phi(\omega_1,\ldots,\omega_n)-M_*\phi(\omega_1,\ldots,\omega_n)|\leq
\frac{1}{n} \sum_{j=1}^n\var_j{\phi} =\mathit{o}(1).$$
\end{proof}

Since we are assuming that $\log |T'(x)|$ has variations uniformly tending to $0$, this lemma has an immediate consequence on the size of basic intervals.
\begin{lem}\label{cylindersize}
For any $\omega\in\Sigma$
$$|\log(\diam(C_n(\omega)))-nA_n(-\log |T'\circ \Pi(\omega)|)|=\mathit{o}(n).$$
\end{lem}

\begin{proof}
This can be proved straightforwardly since  by the mean value theorem we have
$$\log(\diam(C_n(\omega))=nA_n(-\log |T'\circ \Pi(\tau)|)$$
for some $\tau\in\Sigma$ such that $(\tau_1,\ldots,\tau_n)=(\omega_1,\ldots,\omega_n)$. We can then apply Lemma \ref{ancyl} to  $\phi= \log |T'|$ which was assumed to have variations tending uniformly to $0$.
\end{proof}

Now it is  time  to refer to the notion of pressure of a potential. If $\phi:\Lambda\to\R$ is a function with variations uniformly tending to $0$ then we define its pressure by
$$P(\phi)=\sup_{\mu\in \mathcal{M}(T)}\left\{h(\mu)+\int\phi\text{d}\mu:\int\phi\text{d}\mu>-\infty\right\}.$$
This can be alternatively stated as (see \cite{MUbook},  p.\,7)
\begin{equation}\label{pressure2}
P(\phi)=\lim_{n\rightarrow\infty}\frac{1}{n}\log\sum_{|\omega|=n}e^{S_n(\phi\circ \Pi(\overline{\omega}))}.
\end{equation}
Notice that it is possible that $P(\phi)=\infty$.
%{\color{red} ($P(0)=\infty$ ?)}

Finally we prove some important results regarding the relationship between the topological pressure and $s_{\infty}$.
Observe that $t \mapsto P(-t\log |T'|)$ is decreasing because $\log |T'(x)|>0$.

\begin{lem}\label{criticalvalue}
We have
$$s_{\infty}=\inf\left\{t\geq 0:P(-t\log |T'|)<\infty\right\}.$$
\end{lem}
\begin{proof}
For convenience we will let
$$\psi(x)=-\log |T_l'(x)|,
     \quad G(x)=\log\diam(I_l)\text{ where }l=\min\{j:x\in I_j\}.$$
      To complete the proof simply note that if $P(tG)<\infty$ or $P(t\psi)<\infty$  then by (\ref{pressure2}) we have $$|P(tG)-P(t\psi)|\leq t\var_1(\psi).
      $$
\end{proof}
\begin{lem}\label{seqmeas}
There exists a sequence of $T$-invariant probability measures $\{\mu_n\}_{n\in\N}$ such that
%\begin{enumerate}
%\item
$$\lim_{n\rightarrow\infty}\lambda(\mu_n)=\infty, \quad    %$
%\item
%$
\lim_{n\rightarrow\infty}\frac{h(\mu_n)}{\lambda(\mu_n)}=s_\infty.$$
%\end{enumerate}
\end{lem}
\begin{proof} We suppose $s_\infty>0$ and leave the easy case $s_\infty =0$ to readers.
We start by fixing $\epsilon>0$ and noting that for any $T$-invariant measure $\mu$ such that
$\frac{h(\mu)}{\lambda(\mu)}\geq s_{\infty}+2\epsilon$ we have
\begin{eqnarray*}
P(-(s_{\infty}+\epsilon)\log |T'|)&\geq& h(\mu)-(s_{\infty}+\epsilon)\lambda(\mu)\\
&\geq&\epsilon\lambda(\mu)
\end{eqnarray*}
and so $\lambda(\mu)\leq \frac{P(-(s_{\infty}+\epsilon)\log |T'|)}{\epsilon}$.

We now take two sequences $\{t_n\}_{n\in\N}$ and $\{k_n\}_{n\in\N}$ such that for each $n$, $t_n<s_{\infty}$,
$\lim_{n\to\infty}t_n=s_{\infty}$ and $\lim_{n\to\infty}k_n=\infty$.
Since for all $n$ we have that $P(-t_n\log |T'|)=\infty$, by variational principle, we can find a sequence of $T$-invariant measures $\mu_n$ such that 
\[h(\mu_n)- t_n \lambda(\mu_n) \gg1,\]
and hence
$\frac{h(\mu_n)}{\lambda(\mu_n)}> t_n$. Furthermore, by the fact that $\lambda(\mu)\geq h(\mu)$, we can have $$\lambda(\mu_n)\geq k_n.$$ However, for any $\epsilon>0$, if $k_n\geq \frac{P(-(s_{\infty}+\epsilon)\log |T'|)}{\epsilon}$ then $\frac{h(\mu_n)}{\lambda(\mu_n)}\leq s_{\infty}+2\epsilon$. So,
$$\lim_{n\to\infty}\frac{h(\mu_n)}{\lambda(\mu_n)}=s_{\infty}.$$
\end{proof}

\section{Tools}

It will be useful for us to describe in some details the main
tools we are going to use. They are already used in the literature in
the finite symbolic case, but in this paper we are working with
infinitely many symbols and this introduces some minor changes. We
remind the reader that $(\Sigma, \sigma)$ is the full shift on one-sided
symbolic space over an infinite alphabet.

\subsection{Bernoulli approximation}

In this section we will present a process of using sets of cylinders to define Bernoulli type  ergodic measures. This is a similar idea to the Misurewicz's proof of the variational principle but here we also exploit the structure of the symbolic space. Since we are in a non-compact setting, an added complication is  that  weak* limits of measures will not always exist.

Let $\phi:\Sigma\to \R$ have variations uniformly tending to $0$. Let $f=\phi\circ\Pi$. We prove the following result.

\begin{prop}\label{bern1}
Let $\epsilon>0$ and $n\in \mathbb{N}$ be fixed. Suppose that
$$\var_n(A_n \phi)\leq\epsilon,\quad \var_n(A_n \log |T'|)\leq\epsilon.$$
For  any set $J\subseteq\N^n$ and any probability vector $\{p_j\}_{j\in J}$  ($0<p_j<1$, $\sum_{j\in J}p_j=1$),
we can find an ergodic $T$-invariant measure $\mu$ such that
\begin{enumerate}
\item
$\int\phi{\rm{d}}\mu\in (\gamma_1-\epsilon,\gamma_1+\epsilon)$
\item
$\lambda(\mu)\in (\gamma_2-\epsilon,\gamma_2+\epsilon)$.
\item
$h(\mu)=1/n\sum_{j\in J} p_j\log p_j$
%({\bf Here $2\epsilon$ should be $\epsilon$?})
\end{enumerate}
where
$$\gamma_1=\frac{1}{n}\sum_{j\in J}p_j S_nf(\overline{j}), \qquad \gamma_2=\frac{1}{n}\sum_{j\in J} p_j\log\diam(\Pi(j)).
$$

\end{prop}
\begin{proof}
For convenience define $\Psi:\Sigma\rightarrow\R$ by $$
\Psi(\omega)=\log |(T^n)'(\Pi\omega)|.
$$
Each $j$ in $J$ defines  a cylinder.
We start by defining a $\sigma^n$-invariant Bernoulli measure $\nu_n$ on $\Sigma$ by assigning each cylinder $j\in J$ the weight $p_j$. This measure will satisfy
\begin{enumerate}
\item
$\frac{1}{n}\int S_nf\text{d}\nu_n\in(\gamma_1-\epsilon,\gamma_1+\epsilon)$
\item
$\frac{1}{n}\int S_n(\log T'\circ\Pi)\text{d}\nu_n\in(\gamma_2-\epsilon,\gamma_2+\epsilon)$
\item
$h(\nu_n,\sigma^n)=-\sum_{j\in J}p_j\log p_j$.
\end{enumerate}
Then define a $\sigma$-invariant measure

\[
\nu=\frac{1}{n}\sum_{l=0}^{n-1}\nu_n\circ\sigma^{-l}.
\]
{Since the measure $\nu_n$ is $\sigma^n$-ergodic,
$\nu$ is $\sigma$-ergodic.} By straightforward calculations and
Abramov's formula for entropy (see \cite{PU}, Theorem 2.4.6 page 32), the above three formulas can be
written for $\nu$ as
\begin{enumerate}
\item
$\int f\text{d}\nu\in(\gamma_1-\epsilon,\gamma_1+\epsilon)$,
\item
$\int\log T'\circ\Pi\text{d}\nu\in(\gamma_2-\epsilon,\gamma_2+\epsilon)$,
\item
$h(\nu,\sigma)=-\frac{1}{n}\sum_{j\in J}p_j\log p_j$.
\end{enumerate}
To finish the proof we simply let $$
\mu=\nu\circ\Pi^{-1}.
$$
\end{proof}
We will use this proposition in two ways. One is to construct measures from sets of cylinders where the Birkhoff averages for certain potentials will be the same. The other is to approximate invariant measures with ergodic measures.

Let $k\in\N$. For $\gamma=(\gamma_1,\ldots,\gamma_k)\in\R^k$, denote by $\Sigma(\gamma)$ the following set of cylinders in $\Sigma$
$$\big\{[\omega_1,\ldots,\omega_n]:A_n\phi_i(\Pi\omega)\in (\gamma_i-\epsilon, \gamma_i+\epsilon),\forall\omega\in [\omega_1,\ldots,\omega_n], \forall 1\leq i\leq k\big\}.$$

\begin{cor}\label{bern2}
Fix $k\in\N$, $\gamma=(\gamma_1,\ldots,\gamma_k)\in\R^k$ and $n\in\N$. If there exists $s$ such that
$$\sum_{\Sigma(\gamma)}\diam([\omega_1,\ldots,\omega_n])^{s}=1,$$
and $$K:=-\sum_{\Sigma(\gamma)}\diam([\omega_1,\ldots,\omega_n])^{s}\log\diam([\omega_1,\ldots,\omega_n])<\infty.$$
Then there exists a $T$-invariant ergodic measure $\mu$ such that
\[
     \int\phi_i {\rm{d}}\mu\in (\gamma_i-\epsilon, \gamma_i+\epsilon), \ \forall 1\leq i\leq k,  \quad \text{and} \quad \left|\frac{h(\mu)}{\lambda(\mu)}-s\right|\leq\frac{\epsilon}{K-\epsilon}.
\]
\end{cor}

\begin{proof}
We simply apply Proposition \ref{bern1} with $J=\Sigma(\gamma)$ and with $\diam([\omega_1,\ldots,\omega_n])^{s}$
as probabilities.
\end{proof}

\begin{cor}\label{bern3}
If there exists a $T$-invariant measure $\mu$ and a vector $\gamma\in\R^{k}$ ($k\in\N$) such that
$$ \lambda(\mu)<\infty;\quad  \forall 1\leq i\leq k, \int\phi_i\text{d}\mu=\gamma_i,$$
then there exist strictly increasing sequences of integers
$\{q_{\ell}\}, \{n_{\ell}\}$, and a sequence of $T^{n_{\ell}}$-invariant Bernoulli
measures $\{\mu_{\ell}\}$ supported on $\Pi(\Sigma_{q_{\ell}})$ such that

\begin{enumerate}
\item $\lim_{n\rightarrow\infty}\int
A_n\phi_i\text{d}\mu_{\ell}=\gamma_i$ for $1\leq i\leq k$, \item
$\lim_{{\ell}\rightarrow\infty}h(\mu_{{\ell}},T^{n_{\ell}})=h(\mu),$ \item
$\lim_{{\ell}\rightarrow\infty}\lambda(\mu_{\ell},T^{n_{\ell}})=\lambda(\mu)$.
\end{enumerate}
\end{cor}

\begin{proof}
Take such an invariant measure $\mu$. For any $\epsilon>0$ we can
find $N\in\N$ and $q\in \N$ such that for any $n\geq N$
\begin{enumerate}
\item $\var_n\{A_n(\log T'\circ\Pi)\}\leq\epsilon$, \item For each
$1\leq i\leq k$, $\max_i\{\var_n(A_n\phi_i)\}\leq\epsilon$, \item For
each $1\leq i\leq k$,
$$\left|\sum_{\omega_1,\ldots,\omega_n}\tilde{\mu}(\Pi[\omega_1,\ldots,\omega_n])A_n\phi_i(\Pi(\overline{\omega_1,\ldots,\omega_n}))-\gamma_i\right|\leq\epsilon,$$
\item
$\left|\sum_{\omega_1,\ldots,\omega_n}\tilde{\mu}(\Pi[\omega_1,\ldots,\omega_n])\log\diam([\omega_1,\ldots,\omega_n])-n\lambda(\mu)\right|\leq
n\epsilon,$ \item
$\left|\sum_{\omega_1,\ldots,\omega_n}\tilde{\mu}(\Pi[\omega_1,\ldots,\omega_n])\log\tilde{\mu}(\Pi[\omega_1,\ldots,\omega_n])-n
h(\mu)\right|\leq n\epsilon,$
\end{enumerate}
where in points (3)-(5) the sums are taken over all words
$\omega_1\ldots\omega_n \in \{1,\ldots,q\}^n$ and

\[
\tilde{\mu}(\Pi[\omega_1,\ldots,\omega_n]) = \frac
{\mu(\Pi[\omega_1,\ldots,\omega_n])}
{\sum_{\omega_1,\ldots,\omega_n}\mu(\Pi[\omega_1,\ldots,\omega_n])}.
\]

We can now apply the first part of the proof of Proposition
\ref{bern1} to construct our sequence of measures. We could go on
to get a sequence of $T$-ergodic measures. However, these
$T^{n_{\ell}}$-ergodic measures $\mu_{\ell}$ will actually be more useful
for our purposes.
\end{proof}

\subsection{W-measures}
The main tool to prove the lower bound of our Theorems will be to use the technique of $w$-measures used in \cite{GR}. This involves using a sequence of ergodic measures to define a new measure which we will use to calculate the lower bound for the dimension.
\begin{thm} \label{w}
Let $\{{\mu_j}\}_{j=1}^\infty$ be a sequence of \,  $T$-invariant
measures of finite entropy such that the following limits exist
     $$\gamma_i =
\lim_{j\to\infty} \int \phi_i \text{d}{\mu_j}, \quad \forall i \in \mathbb{N}.$$
Then for  $\underline{\gamma} = (\gamma_i)_{i\in \N}$ we have
\[
\dim \Lambda_{\underline{\gamma}} \geq \limsup \frac {h(\mu_j)} {\lambda(\mu_j)}.
\]
\end{thm}
\begin{proof}

This statement is analogous to the one proven in \cite[Proposition 9,
Theorem 3]{GR} in the special case: it was a finite iterated
function system, the measures $\mu_j$ were Gibbs and there was
only one potential $\phi=\log |T'|$. The proof of the general
statement is analogous, but there are some changes so we rewrite
it.

By choosing a subsequence we can freely assume that
$h(\mu_j)/\lambda(\mu_j)$ have a limit.

To begin, we are not going to use the measures $\mu_j$
directly. Fix a sequence $\varepsilon_j\to 0$,  by Corollary
\ref{bern3}, for each $j$, there exist an integer $n_j$ and a Gibbs (even Bernoulli) $T^{n_j}$-invariant measure $\mu'_j$ such that
\begin{enumerate}
\item $\left|\int A_n\phi_i\text{d}\mu_j'-\gamma_i\right|< \varepsilon_j/2$ for $1\leq i\leq j$,
%({\bf Here $k$ should be $j$?})
\item
$|h(\mu'_{j},T^{n_j})-h(\mu_j)|<\varepsilon_j/2,$
\item
$|\lambda(\mu'_j,T^{n_j})-\lambda(\mu_j)|<\varepsilon_j/2$.
\end{enumerate}
     Then let
\begin{equation} \label{nazwa}
{\eta_j}=\frac 1 {n_j} \sum_{l=0}^{n_j-1} \mu'_j\circ \Pi\circ \sigma^{-l}.
\end{equation}
The family
$\{{\eta_j}\}_{j=1}^\infty$ has the following properties:
\begin{itemize}
     \item[--] $h(\eta_j) = \frac 1 {n_j} h(\mu'_j; \sigma^{n_j})$,
     \item[--] each measure ${\eta_j}$ is supported on a
symbolic space $\Sigma_{q_j}$ with only finitely many symbols, the sequence $\{q_j\}$ is in
general unbounded. Note that $\Sigma_{q_j}$ is compact, hence each
$f_i=\phi_i\circ \Pi$ is bounded on $\Sigma_{q_j}$,
%   \item[--] each measure ${\eta_j}$ is of a form
%\begin{equation} \label{nazwa}
%{\eta_j}=\frac 1 {n_j} \sum_{l=0}^{n_j-1} \mu_j\circ \sigma^l,
%\end{equation}
%where $\mu_j$ is a Gibbs (in fact, Bernoulli) measure for
%$\sigma^{n_j}$ and $n_j$ is a positive integer,
     \item[--]
\[
\left| \frac {h(\eta_j)} {\lambda({\eta_j})} -
\frac {h({\mu_j})} {\lambda(\mu_j)} \right| \leq \varepsilon_j,
\]
     \item[--] for all $1\leq i \leq j$
\[
\left| \int f_i \text{d}{\eta_j} - \int \phi_i \text{d}\mu_j \right| \leq
\varepsilon_j.
\]
\end{itemize}
%The construction of the measures ${\mu_j}$ is done by Corollary
%\ref{bern3}, the measures $\{\eta_j\}$ are then obtained as in
%formula \eqref{nazwa}.
%
\medskip

Let $\{m_j\}$ be a fast increasing sequence of integers (in the
following we will provide further conditions). We will construct a
probability measure $\eta$ supported on $\Sigma$ by defining it
on a family of cylinders, which has a product structure.

First, on all cylinders of level $m_1$ we define
\[
\eta([\omega_1,\ldots,\omega_{m_1}]) =
\eta_1([\omega_1,\ldots,\omega_{m_1}]).
\]
Then, in an inductive step, having the measure defined on cylinders of level
$m_{j-1}$, we subdivide it on their subcylinders of level $m_j$ by
the following formula:
\[
\eta([\omega_1,\ldots,\omega_{m_j}]) =
\eta([\omega_1,\ldots,\omega_{m_{j-1}}]) \cdot
\eta_j([\omega_{m_{j-1}+1},\ldots,\omega_{m_j}]).
\]
We assume that $$
        m_1\gg n_1, \qquad (m_j-m_{j-1})\gg n_j.$$
         Note that at
each step of construction the measure is defined on a symbolic space with finitely many symbols.
Denote
\[
L_n(\omega)=\frac 1 n \log |(T^n)'(\Pi \omega)|
\quad
\text{and}
\quad
M_n(\omega)=-\frac 1 n \log \eta([\omega_1,\ldots, \omega_n]).
\]
%{\color{red}(The proposition here, which  brakes the proof of Theorem 3.4., is taken off. )}
%\begin{prop} \label{wmain}
We claim the following: we can choose $\{m_j\}$ such that
\begin{equation}\label{EQ0}
-\log\lambda_{q_{j+1}} < \varepsilon_{j+1} m_j
\end{equation}
where  $\lambda_j$ is the maximal contraction ratio
of map $T_j$ and that
the following properties are
satisfied for any $j$ and for all points $\omega$ in a positive
$\eta$-measure set $A\subset\Sigma$:  for all $1\leq i \leq j$ and $m_j \leq n < m_{j+1}$ we have
\begin{equation}\label{EQ1}
M^*f_i(\omega_1,\ldots,\omega_n)-M_*f_i(\omega_1,\ldots,\omega_n)
\leq \varepsilon_j,
\end{equation}
\begin{equation}\label{EQ2}
\left| A_nf_i(\omega) - \frac {m_j} n \int f_i {\rm{d}}{\eta_j} -\frac
{n-m_j} n \int f_i {\rm{d}}{\eta_{j+1}}\right| \leq \varepsilon_j,
\end{equation}
\begin{equation}\label{EQ3}
\left| L_n(\omega) - \frac {m_j} n\lambda({\eta_j}) -\frac
{n-m_j} n \lambda({\eta_{j+1}})\right| \leq \varepsilon_j,
\end{equation}
    \begin{equation}\label{EQ4}
\left| M_n(\omega) - \frac {m_j} {n} h(\eta_j)
-\frac {n-m_j} {n} h(\eta_{j+1}) \right|
\leq \varepsilon_j.
\end{equation}

%{\color{red} ($\tilde{A}_nf_i$ is not defined, notation being changed.)}

Let us prove the last four expressions. The formula (\ref{EQ1}) follows from
Lemma \ref{ancyl} provided all $m_j$ are big enough. The other
three expressions are the main part. Note that (\ref{EQ3}) and
(\ref{EQ4}) are actually special cases of (\ref{EQ2}).
$L_n(\omega)$ is (by bounded distortion) approximately a
partial Cesaro average of the function $\log |T'|$. Similarly, while $\eta_j$ is not a Gibbs measure, $\mu_j'$ is (for $T^{n_j}$). Hence,
$\frac 1 n (n M_n(\omega) - m_j M_{m_j}(\omega))$ is (by
Gibbs property) approximately a partial Cesaro average of the
potential of the Gibbs measure $\mu'_{j+1}$ (average under iterations of $T^{n_{j+1}}$). For this reason, we will
provide a detailed proof of the formula (\ref{EQ2}) only and the formulas
(\ref{EQ3}) and (\ref{EQ4}) can be proved analogously.

Applying the Birkhoff Ergodic Theorem to the measure $\eta_1$ and the function $f_1$, we get that

\begin{equation} \label{stat1}
\left| A_{m_1}f_1(\omega)-\int f_1 {\rm{d}}\eta_1\right| \leq \frac
{\varepsilon_1} 2
\end{equation}
on a set of $\eta_1$-measure $1-\delta_1$, where $\delta_1$ can be
chosen arbitrarily small if $m_1$ is sufficiently big. The next
statement we will need is that

\begin{equation} \label{stat2}
n\left|  A_n f_1(\sigma^{m_1}(\omega)) - \int f_1
{\rm{d}}\eta_2 \right| \leq \frac {m_1 \varepsilon_1} 2 +
n\varepsilon_2
\end{equation}
for all $n\geq 1$ for a set of $\omega$ of $\eta_2$-measure
$1-\tilde{\delta}_1$ (more precisely, we will only need this
statement for $1\leq n \leq m_2-m_1$, but it is important that we
can choose arbitrarily big $m_2$ and the statement will still be true). It follows from the Central Limit Theorem (see
\cite[Thm 5.7.1]{PU}) for the measure $\eta_2$ that for any continuous
$f$ and for big $n$
\[
\left|  A_n f(\omega) - \int f {\rm{d}}\eta_2
\right| < \varepsilon
\]
for all $\omega$ except a subset of measure approximately
$\exp(-cn\varepsilon^2)$. Hence, $\tilde{\delta}_1$ can be chosen
arbitrarily small, provided $m_1 \varepsilon_1$ is big enough (how
big is big enough will depend on $\varepsilon_2$).
%{\bf (A
%little more details: how to get the two terms on the RHS of
%\ref{EQ2}?) (I do not understand, do you mean \eqref{stat2}? it is
%a standard estimation: for every $\phi$, for every $\epsilon$ and
%every $a>0$ one can choose $b$ such that with probability greater
%than $1-\epsilon$ for all $n>0$ the deviation $|S_n \phi -n\cdot
%E\phi|$ is smaller than $b+na$)}

We continue in an inductive way. By the Birkhoff Ergodic Theorem we have
\begin{equation} \label{stat3}
\left| A_{m_j}f_i(\omega)-\int f_i \text{d}{\eta_j}\right| \leq \frac
{\varepsilon_j} 2
\end{equation}
for all $1\leq i \leq j$ on a set of $\eta$-measure $1-\delta_j$,
where $\delta_j$ can be chosen arbitrarily small provided $m_j$ is
sufficiently big and sufficiently big in comparison with
$m_{j-1}$. By the Central Limit Theorem

\begin{equation} \label{stat4}
n\left|  A_nf_i(\sigma^{m_j}(\omega)) - \int f_i
\text{d}{\eta_{j+1}} \right| \leq \frac {m_j \varepsilon_j} 2 +
n\varepsilon_{j+1}
\end{equation}
for all $1\leq i \leq j$ and $n\geq 1$ for a set of $\omega$ of
${\eta_{j+1}}$-measure $1-\tilde{\delta}_j$, where
$\tilde{\delta}_j$ can be chosen arbitrarily small, provided $m_j
\varepsilon_j$ is big enough. Combining \eqref{stat1},
\eqref{stat2}, \eqref{stat3} and \eqref{stat4} we get (\ref{EQ2})
true on a set $A$ of $\eta$-measure at least
$1-\sum\delta_j-\sum\tilde{\delta}_j$, which can be chosen
arbitrarily close to 1.
%{\bf why
%$1-\sum\delta_j-\sum\tilde{\delta}_j$? some detailed estimations
%from the structure of $\eta$. because $\delta_j$ and
%$\tilde{\delta}_j$ was the upper bound on the measure of points
%where something goes wrong between $m_{j-1}$ and $m_j$}

%The point v) is always true if $m_j$ are sufficiently big.

Let $\eta_A$ be the restriction of $\eta$ to $A$. By (\ref{EQ1}) and (\ref{EQ2}), we have
$$A\subset \Lambda_{\underline{\gamma}}.$$
   On the other hand, for
all $m_j<n\leq m_{j+1}$, $A$ is contained in a union of $n$th level
cylinders, each of size at least $$
r_n:=\exp\left(-m_j\lambda({\eta_j})-
(n-m_j)\lambda({\eta_{j+1}})-n\varepsilon_j\right)$$
    (by (\ref{EQ3})) and of $\mu$-measure at most $$
    c_n:=\exp\left(-m_j h({\eta}_j) - (n-m_j)
h({\eta}_{j+1}) + n\varepsilon_j\right)$$ (by
(\ref{EQ4})).
%Denote
%\[
%r_n=\exp\left( -m_j\lambda({\eta_j})-
%(n-m_j)\lambda({\eta_{j+1}})-n\varepsilon_j\right).
%\]
According to (\ref{EQ0}), we have
\[
|\log r_{n+1} - \log r_n| \leq \varepsilon_j |\log r_n|/n.
\]

For any $\omega\in A$, the ball $B_{r_n}(\omega)$ intersects $A$ at
most two $n$th level cylinders. Hence
\[
{\eta}_A(B_{r_n}(\omega)) \leq 2 c_n
%\exp\left( -\frac 1 {k_j} m_j
%h(\tilde{\eta}_j;\sigma^{k_j}) - \frac 1 {k_{j+1}} (n-m_j)
%h(\tilde{\eta}_{j+1};\sigma^{k_{j+1}}) + n\varepsilon_j\right).
\]
By Frostman's Lemma,

\[
\dim \Pi(A) \geq \liminf \frac {h({\eta}_j)} {\lambda({\eta_j})} = \liminf \frac {h(\mu_j)} {\lambda(\mu_j)}.
\]
Recall that at the beginning, we assume that
$h(\mu_j)/\lambda(\mu_j)$ have a limit. The proof is then completed.
\end{proof}

\section{Proof of Theorem \ref{main0}}

The proof is divided into the following three propositions. Recall that
$$\Lambda_{\underline{\gamma}}=\{x\in\Lambda:\lim_{n\rightarrow\infty}A_n\phi_i(x)=\gamma_i\text{ for all }i\in\N\}.$$

\begin{prop} \label{lem:A}
If $\underline{\gamma}\notin Z$ then $\Lambda_{\underline{\gamma}}=\emptyset$.
%$$\{x\in\Lambda:\lim_{n\rightarrow\infty}A_n\phi_i(x)=\gamma_i\text{ for all }i\in\N\}=\emptyset.$$
\end{prop}

\begin{proof}
Given $\underline{\gamma}$, assume
that there exists $x\in\Lambda$ such that for all $i\in\N$
$\lim_{n\rightarrow\infty}A_n\phi_i(x)=\gamma_i$. Let
$\omega\in\Sigma$ satisfy $\Pi\omega=x$. If we fix $\epsilon>0$
and $k\in\N$ then we can find $N$ such that for all $n\geq N$ we have
$$\sup_{1\leq i\leq
k} |A_n\phi_i(x)-\gamma_i|\leq\epsilon/2,$$
$$\sup_{1\leq i\leq
k}\sup_{x,y\in\Pi([\omega_1,\ldots,\omega_n])}|A_n\phi_i(x)-A_n\phi_i(y)|\leq\epsilon/2.$$
We then let $\nu$ be the shift invariant measure on $\Sigma$
defined on the periodic orbit
$\overline{(\omega_1,\ldots,\omega_n)}$. If we let
$\mu=\nu\circ\Pi$ %({\bf Can we define like this? If $\Pi$ is not inversible?})
then we have that $\mu$ is $T$-invariant
and that $\left|\int\phi_i\text{d}\mu-\gamma_i\right|\leq\epsilon$
for each $1\leq i\leq k$. This completes the proof.
\end{proof}

In what follows, we will restrict ourself to the case $\underline{\gamma}\in Z$.
\begin{prop}
If $\underline{\gamma}\in Z$ then
\(
\dim \Lambda_{\underline{\gamma}} \geq \alpha_1(\gamma).
\)
\end{prop}
\begin{proof}
It
follows immediately from Theorem \ref{w}.
\end{proof}

\begin{prop}
If $\underline{\gamma} \in Z$ then \( \dim
\Lambda_{\underline{\gamma}} \leq \alpha_2(\gamma). \)
\end{prop}
\begin{proof}
Let $\tilde{s}=\dim
\Lambda_{\underline{\gamma}}=\dim
(\Lambda_{\underline{\gamma}}\backslash E)$. Given $\varepsilon>0$, for any
covering of $\Lambda_{\underline{\gamma}}\backslash E$ with intervals $E_j$ of
lengths $|E_j|<\delta$ we will have
\[
\sum |E_j|^{\tilde{s}-\varepsilon} > N(\delta)
\]
with $N(\delta)\to\infty$ as $\delta\to 0$. In particular, if we
choose a covering of $\Lambda_{\underline{\gamma}}$ with $n$th
level basic intervals, the corresponding sum $\sum |\Pi[\omega_1,\ldots,
\omega_n]|^{\tilde{s}-\varepsilon}$ will be greater than $1$
provided $n$ being big enough. If this summand is infinite, we can
choose a finite subfamily of $n$th level basic intervals intersecting
$\Lambda_{\underline{\gamma}}$ such that sum of their diameters in
power $\tilde{s}-\varepsilon$ is still greater than $1$. We can then
choose a different exponent $s > \tilde{s}-\varepsilon$ for which
this sum is equal to $1$.

By Lemma \ref{ancyl}, for any $k$ for sufficiently big $n$ if an
$n$th level cylinder intersects $\Lambda_{\underline{\gamma}}$
then
\[
|A_n\phi_i(\omega) - \gamma_i|<\varepsilon
\]
for all $i\leq k$ and for all $\omega$ in this cylinder.

We can now apply Proposition \ref{bern1} and Corollary \ref{bern2} to construct an ergodic
measure $\nu$ with respect to the shift acting on finitely many symbols (hence, of finite
entropy), and then a $T$-invariant ergodic measure $\mu$ satisfying
\[
\left|\int \phi_i \text{d}\mu - \gamma_i\right|<2\varepsilon, \quad \left|\frac{h(\mu)}{\lambda(\mu)}-s\right|\leq\frac{2\epsilon}{K-2\epsilon}
\]
for all $1\leq i \leq k$.
By the formula $\dim \mu = h(\mu)/\lambda(\mu)$ the proof of the upper bound
in Theorem \ref{main0} is completed.

%of Hausdorff dimension at least $\tilde{s}-2\varepsilon$,
%{\color{red} (The formula  is used but not referred)} This ends the proof of the upper bound
%in Theorem \ref{main0}.
\end{proof}

%{\color{red} (Nothing is said about $\alpha_1=\alpha_2$.)}

\section{Proof of Theorem \ref{main}}

From now on we will assume that each function $\phi_i$ is bounded above and below.
Recall that
$$
     \alpha_1(\underline{\gamma})=
     \lim_{\varepsilon\to 0}
\lim_{k\to\infty}
\sup_{\mu\in\mathcal{M}(T)}\left\{\frac{h(\mu)}{\lambda(\mu)}:\left|\int\phi_i
\text{d}\mu-\gamma_i\right| < \varepsilon\ \forall i\leq k, \ h(\mu)<\infty\right\}.
$$
$$
\alpha_3(\underline{\gamma})
=\sup_{\mu\in\mathcal{M}(T)}\left\{\frac{h(\mu)}{\lambda(\mu)}:\int\phi_i
\text{d}\mu=\gamma_i \ \forall i\in\N, \ h(\mu)<\infty\right\}.
$$
%{\color{red} (the notation $\alpha_1, \alpha_3$ is used below without repeating .)}

We will first show that for all $\underline{\gamma}\in Z$ we have (see Lemma  \ref{lem5-1})
$$
\alpha_1(\underline{\gamma}) \geq s_{\infty}.$$
As Theorem \ref{main0} is already proven, we
    then have
$$\dim \Lambda_{\gamma}=\max\left\{s_{\infty}, \ \alpha_1(\underline{\gamma}) \right\}.$$
Then we will show that (see Proposition \ref{semicont})
$$
\alpha_1(\underline{\gamma})> s_{\infty}
\Rightarrow
\alpha_1(\underline{\gamma})= \alpha_3(\underline{\gamma})
$$
It will follow that
$\dim \Lambda_{\gamma}=\max\left\{s_{\infty}, \ \alpha_3(\underline{\gamma}) \right\}$.

Let us first prove  the following Lemma.
\begin{lem}\label{lem5-1}
Let $\underline{\gamma}\in Z$, $k\in\N$, $\epsilon>0$ and $\mu\in{ \mathcal{M}(T)}$ such that $$
\lambda(\mu)<\infty,\qquad \sup_{1\le i \le k}\left|\int\phi_i{\rm{d}}\mu-\gamma_i\right|\leq\varepsilon.$$
    There then exists a measure $\nu\in{\mathcal{M}(T)} $ such that $$
    \frac{h(\nu)}{\lambda(\nu)}\geq s_{\infty}-\epsilon, \qquad \sup_{1\le i \le k}\left|\int\phi_i{\rm{d}}\mu-\gamma_i\right|\leq 2\varepsilon.$$
\end{lem}
\begin{proof}
Let $A=\sup_{1\leq i\leq k} \sup_{x\in\Lambda} |\phi_i(x)|$.
By Lemma \ref{seqmeas} we can find a sequence of $T$-invariant measures $\mu_n$ such that $\lim_{n\rightarrow\infty}\lambda(\mu_n)=\infty$ and $\frac{h(\mu_n)}{\lambda(\mu_n)}\geq s_{\infty}-\frac{\epsilon}{2}$ for each $n$.
Consider the measure $$
\nu_n=(1-\frac{\epsilon}{A})\mu+\frac{\epsilon}{A}\mu_n.$$ Then we have that for each $1\leq i\leq k$
$$\left|\int\phi_i\text{d}\nu_n-\gamma_i\right|\leq \left|\int\phi_i\text{d}\mu-\gamma_i\right|+\left|\int\phi_i\text{d}\mu-\int\phi_i\text{d}\nu_n\right|\leq 2\epsilon.$$
Furthermore
\begin{eqnarray*}\liminf_{n\rightarrow\infty}\frac{ h(\nu_n)}{\lambda(\nu_n)}&=&\liminf_{n\rightarrow\infty}\frac{(1-\epsilon/A)h(\mu)+\epsilon/A h(\mu_n)}{(1-\epsilon/A)\lambda(\mu)+\epsilon/A\lambda(\mu_n)}\\
&=&\liminf_{n\rightarrow\infty}\frac{h(\mu_n)}{\lambda(\mu_n)}\\
&\geq& s_{\infty}-\frac{\epsilon}{2}.\end{eqnarray*}
This completes the proof.
\end{proof}

Thus we can conclude that for all $\underline{\gamma}\in Z$, we have
$\alpha_1(\underline{\gamma})\geq s_{\infty}$.
%and the lower bound for Theorem \ref{main} immediately follows. To complete the proof of the upper bound we simply need the following %proposition.
\begin{prop}\label{semicont}
Let $\underline{\gamma}\in Z$.  If
$\alpha_1(\underline{\gamma}) > s_{\infty}$, we have $\alpha_1(\underline{\gamma}) =\alpha_3(\underline{\gamma}).$
\end{prop}

The proof of the proposition is lengthy and it is presented in the next section.

\section{Proof of  Proposition \ref{semicont}}

The assertion of Proposition \ref{semicont} will follow
immediately from the following statement.
\begin{prop} \label{upper}
Given $\underline{\gamma}\in Z$ and a sequence of invariant
measures $\mu_j$ such that
\begin{itemize}
     \item[--] $h(\mu_j)/\lambda(\mu_j)>s_\infty + \delta$ for some
$\delta>0$,
     \item[--] $\int \phi_i {\rm{d}}\mu_j \to \gamma_i$ for all $i\in\N$,
\end{itemize}
there exists an invariant measure $\mu$ satisfying
$$\frac{h(\mu)}{\lambda(\mu)} = \limsup \frac{h(\mu_j)}{\lambda(\mu_j)} \quad {\rm and} \quad \int
\phi_i {\rm{d}}\mu = \gamma_i \ \forall i\in\N.$$
\end{prop}

To prove the statement we will consider the locally constant potentials $\psi_k$ defined by
$$\psi_k(x)=\frac{1}{k}\sup_{y\in C_k(x)}\log |(T^k)'(y)|.$$
We then have the following straightforward lemma.
\begin{lem}\label{approx}
For any $\mu\in\mathcal{M}(T)$ such that $\lambda(\mu)<\infty$ we have
$$\left|\lambda(\mu)-\int \psi_k{\rm{d}}\mu\right|=\mathit{o}(1).$$
\end{lem}
\begin{proof}
This follows simply because the variations of $\log |T'(x)|$ tend uniformly to $0$.
\end{proof}

We will first prove an analogous statement to Proposition
\ref{upper} for $\psi_k(x)$ and then use Lemma \ref{approx} to
deduce Proposition \ref{upper}. For convenience for
$\mu\in\mathcal{M}(T)$ we will let
$\xi_k(\mu)=\int\psi_k\text{d}\mu$.
\begin{lem}\label{kthapprox}
Fix any $k\in\N$. Given $\underline{\gamma}\in Z$ and a sequence of invariant
measures $\mu_j$ such that
\begin{itemize}
     \item[--] $h(\mu_j)/\xi_k(\mu_j)>s_\infty + \delta$ for some
$\delta>0$,
     \item[--] $\int \phi_i \text{d}\mu_j \to \gamma_i$ for all $i\in\N$,
\end{itemize}
there exists an invariant measure $\mu$ satisfying
$$\frac{h(\mu)}{\xi_k(\mu)} = \limsup_j \frac{h(\mu_j)}{\xi_k(\mu_j)};  \quad \text{and} \quad  \int
\phi_i \text{d}\mu = \gamma_i \ \forall i\in\N. $$
\end{lem}

Note that to prove Lemma \ref{kthapprox} it suffices to prove the
statement for $k=1$ since the statement for general $k$ can then
be deduced by considering the map $T^k$.
%\subsection*{Proof of Lemma \ref{kthapprox}}
The proof of Lemma \ref{kthapprox} will now follow by a series of technical lemmas.

\begin{lem}\label{ub4}
For any $\delta>0$ there is $K(\delta)>0$ such that if $\mu$ is a
$T$-invariant measure and $\frac{h(\mu)}{\xi_1(\mu)}>s_{\infty}
+ \delta$ then $h(\mu)\leq\xi_1(\mu)\leq K(\delta)$.
\end{lem}
\begin{proof}
We fix $t\in\R$ such that $s_{\infty}<t<s_\infty + \delta$.  By the methods from Lemma \ref{criticalvalue} we get $P(-t\psi_1)<\infty$. So by the
variational principle we get $h(\mu)-t\xi_1(\mu)\leq P(-t\psi_1)$.
Since $\frac{h(\mu)}{\xi_1(\mu)}>s_{\infty}
+ \delta$, we have
$$P(-t\psi_1)\geq  (s_\infty + \delta -t)\xi_1(\mu).$$
So, $$\xi_1(\mu)\leq \frac{P(-t\log T')}{s_\infty +
\delta-t}.$$
\end{proof}
Therefore if the hypothesis of Lemma \ref{kthapprox} holds then we
can deduce that the sequence of measures $\{\mu_j\}$ is tight and
so will have at least one limit point $\mu$ which will be a
$T$-invariant probability measure. Moreover by the lower-semi
continuity of $\xi_1(\mu_j)$ (see Lemma 1 in \cite{JMU}), by the simple fact that $h(\mu)\leq\lambda(\mu)$ and the fact that $\lambda(\mu)\leq\xi_1(\mu)$ we know
that $h(\mu)\leq\xi_1(\mu)\leq K$. To finish the proof of
Proposition \ref{upper} we would only need entropy to be upper semi-continuous.

Unfortunately, the entropy is not upper semicontinuous on
${\mathcal{M}}(T)$. We have, however, a limited form of semicontinuity
when we consider entropy divided by Lyapunov exponent, and this will be enough:

\begin{lem}
Let $\{\mu_j\}_{j\in\N}$ be a sequence of measures converging weakly to
$\mu$ and satisfying that
$h(\mu_j)/\xi_1(\mu_j)>s_\infty + \delta$ for some $\delta>0$ and all $j\in\N$. We have
\[
\frac {h(\mu)} {\xi_1(\mu)} \geq \limsup \frac {h(\mu_j)}
{\xi_1(\mu_j)}.
\]
\end{lem}
\begin{proof}
Denote by $\eta_j$ the measure on $\Sigma$ such that $\mu_j=\eta_j\circ\Pi^{-1}$.
We start by choosing a subsequence of $\eta_j$ such that
$h(\eta_j)/\xi_1(\eta_j)$ converges to the maximal possible limit.

Given $q$, consider the projection $\pi_q:\Sigma\to\Sigma_q$
obtained by replacing in a sequence $\omega_1, \omega_2,\ldots$ all symbols
$q+1, q+2,\ldots$ by symbol $q$. The projection of a measure $\nu$
under $\pi_q$ will be denoted by $\nu|_{q}$.

Let us denote $$c_{j,q}=\sum_{k> q} \eta_j([k])$$
$$\tilde{\lambda}_q:=|\log \inf_{x\in \cup_{l=q}^\infty I_l}\{|T'(x)|\}|.$$

Note that
$c_{j,q}$ is uniformly (in $j$) converging to 0 as $q$ increases.
Consider the two partitions:
\[
{\mathcal{A}} = \Big\{[1], [2],\ldots,[q-1], \bigcup_{k=q}^\infty [k]\Big\},
\quad
{\mathcal{B}} =\Big\{\bigcup_{k=1}^q [k], [q+1], [q+2],\ldots \Big\}.
\]
We have
\[
h(\eta_j) = h(\eta_j|{\mathcal{A}}\vee{\mathcal{B}}) \leq h(\eta_j|{\mathcal{A}})
+ h(\eta_j|{\mathcal{B}}).
\]
The former summand is $h(\eta_j|_q)$. The latter can be bounded
from above by the entropy of the corresponding Bernoulli measure.
It has one atom with measure $1-c_{j,q}$ and the other atoms are
cylinders $[k]$ $(k>q)$. Hence,
\begin{align} \label{eqqq}
\begin{split}
h(\eta_j|{\mathcal{B}}) &\leq (1-c_{j,q})|\log (1-c_{j,q})| + c_{j,q}
|\log c_{j,q}| + c_{j,q} h(\nu_{j,q})\\
&\leq c_{j,q} h(\nu_{j,q}) + \varepsilon_0(q),
\end{split}
\end{align}
where $\nu_{j,q}$ is the Bernoulli measure obtained by assigning
on each symbol $k>q$ probability $\eta_j([k])/c_{j,q}$, and
$\varepsilon_0(q)$ converges to $0$ as $q\to \infty$. We know that
$$\xi_1(\nu_{j,q})\geq \log \inf_{x\in \cup_{l=q}^\infty I_l}\{|T'(x)|\}=\tilde{\lambda}_q$$ which must tend to $\infty$ as $q$ goes to $\infty$. Thus by Lemma
\ref{ub4}
\begin{equation} \label{per}
\frac{h(\nu_{j,q})}{\xi_1(\nu_{j,q})} \leq s_\infty +
\varepsilon_1(q)
\end{equation}
for some $\varepsilon_1(q)$ converging to $0$ as $q\to\infty$. At
the same time,
\begin{equation} \label{eqq}
\lambda(\eta_j) \geq \sum_k \eta_j([k])\psi_1(\Pi(\overline{k})) =
\lambda(\eta_{j}|_q) + c_{j,q}(\lambda(\nu_{j,q}) -
\psi_1(\Pi(\overline{q}))).
\end{equation}

As $\xi_1(\eta_j)<\infty$, $c_{j,q} \psi_1(\Pi(\overline{q}))$
must converge to $0$, but this convergence is not uniform. Still,
from the sequence $\{\eta_j\}$ we can choose a subsequence
$\eta_{j_k}$, a sequence $q_l$ and a sequence
$\varepsilon_2(q_l)\to 0$ such that for each $q_l$ we have
\[
\limsup_{j_k} c_{j_k,q_l} \psi_1(\Pi(\overline{q_l})) <
\varepsilon_2(q_l).
\]
Indeed, otherwise we would be able to choose a sequence
$\eta_{j_k}$ such that for some $c>0$ and for any sufficiently big
$q$ we would have
\[
\liminf_{j_k} c_{j_k, q} \psi_1(\Pi(\overline{q})) > c
\]
and that would imply that $\xi_1(\eta_j)=\infty$.

So, finally we get by \eqref{eqqq}, \eqref{per} and
\eqref{eqq} and Lemma \ref{approx} that given $l$, for all $k$ big enough
we have
\begin{equation} \label{eqqq1}
h(\eta_{j_k}) - h(\eta_{j_k}|_{q_l}) < s_\infty \cdot K(j_k, q_l) +
\varepsilon_3(q_l,\delta)
\end{equation}
and
\begin{equation} \label{eqq1}
\xi_1(\eta_{j_k}) - \xi_1(\eta_{j_k}|_{q_l}) > K(j_k, q_l) -
\varepsilon_3(q_l,\delta),
\end{equation}
where $K(j,q)=c_{j,q} \xi_1(\nu_{j,q})> 0$.

Consider now the following diagram:
\begin{center}
\begin{picture}(100,90)(0,-65) \put(5,0){$\eta_{j_k}$}
%\put(25,4){\vector(1,0){40}}
\put(35,1.5){$\dashrightarrow$}
\put(75,0){$\eta$}
    \put(10,-10){\vector(0,-1){30}}
     \put(79,-38){\vector(0,1){30}}
     \put(1,-51){$\eta_{j_k}|_{q_l}$}
\put(75,-51){$\eta|_{q_l}$} %\put(25,-48){\vector(1,0){40}}
\put(35,-48){$\dashrightarrow$}
\end{picture}
\end{center}
By \eqref{eqqq1} and \eqref{eqq1}, given $l$, for $k$ big enough
\[
\frac {h(\eta_{j_k}|_{q_l})} {\xi_1(\eta_{j_k}|_{q_l})} \geq \frac
{h(\eta_{j_k})} {\xi_1(\eta_{j_k})} - \varepsilon(q_l,\delta).
\]
The convergence of $\eta_{j_k}|_{q_l}$ to $\eta|_{q_l}$ takes
place in space of invariant measures of $(\Sigma_{q_l},\sigma)$,
where entropy (and hence $h/\xi_1$) is upper semicontinuous.
Finally, $h(\eta)=\lim h(\eta|_{q_l})$. Taking $\mu=\eta\circ \Pi^{-1}$, we have
\[
\frac {h(\mu)} {\xi_1(\mu)} > \lim \frac {h(\eta_{j_k})}
{\xi_1(\eta_{j_k})} - \varepsilon(q_l).
\]
As we can choose arbitrarily big $q_l$, $\varepsilon(q_l)$ is
arbitrarily small. We are done.
\end{proof}
The statement of Lemma \ref{kthapprox} now follows.

To complete the proof of Proposition \ref{upper} choose a sequence of $T$-invariant
measures $\mu_j$ such that
\begin{itemize}
     \item[--] $h(\mu_j)/\lambda(\mu_j)>s_\infty + \delta$ for some
$\delta>0$,
     \item[--] $\int \phi_i \text{d}\mu_j \to \gamma_i$ for all $i\in\N$.
     \end{itemize}
We choose $\epsilon>0$ sufficiently small such that
$h(\mu_j)/(\lambda(\mu_j)+\epsilon)>s_{\infty}+\delta/2$. We then
choose $k$ sufficiently large such that $\var_k(\log
|T'(x)|)<\epsilon$ and so in particular
$\xi_k(\mu_j)-\lambda(\mu_j)<\epsilon$. Thus
$h(\mu_j)/\xi_k(\mu_j)>s_{\infty}+\delta/2$ and we may apply Lemma
\ref{kthapprox} to show that there exists a $T$-invariant measure
$\mu$ such that $h(\mu)/\xi_k(\mu) = \limsup_{j\to\infty}
h(\mu_j)/\xi_k(\mu_j)$ and $\int \phi_i \text{d}\mu = \gamma_i$
for all $i\in\N$. Moreover
\begin{eqnarray*}\limsup_{j\to\infty}\frac{h(\mu_j)}{\lambda(\mu_j)}&\geq& \limsup_{j\to\infty}\frac{h(\mu_j)}{\xi_k(\mu_j)}=h(\mu)/\xi_k(\mu)\\
&\geq&\frac{h(\mu)}{\lambda(\mu)+\epsilon}\geq
\frac{h(\mu)}{\lambda(\mu)}+\frac{\epsilon h(\mu)}{\lambda(\mu)^2+\epsilon\lambda(\mu)}
\end{eqnarray*}
and Proposition \ref{upper} now easily follows.

This completes the proof of Theorem \ref{main}.
\section{Examples}
We now look at some examples where our results can be applied. We will consider an application to frequency of digits which applies the fact that our level sets are defined using countably many functions.
We then consider two cases which look at possible behaviour when the level set is just determined by one bounded function.
\subsection{Frequency of digits}
There have been many papers on the Hausdorff dimension of sets determined by the frequency of digits for various types of expansion, see for example \cite{bes}, \cite{BSS2}, \cite{E}, \cite{FLM}, \cite{FLMW}, \cite{O2}. Here we show how our results can be applied to give results in this direction in the setting of expanding maps with countably many branches. We take a partition $\{I_i\}_{i\in\N}$ and a map $T$ as in the first section. We define $\phi_i$ to be the characteristic function for the interval $I_i$, that is
$$\phi_i(x)=\chi_{I_i}(x):=\left\{\begin{array}{lll}1&\text{ if }&x\in I_i\\0&\text{ if }&x\notin I_i\end{array}\right.$$
For an infinite vector $\underline{p}=(p_1,p_2,\ldots)$ where $\sum_{i=1}^{\infty}p_i\leq 1$  let
$$\Lambda_{\underline{p}}=\{x\in\Lambda:\lim_{n\rightarrow\infty} A_n\phi_i(x)=p_i\text{ for all }i\in\N\}.$$

The assumptions of Theorem \ref{main} are all satisfied and it is easy to see that all such  $\underline{p}$ belong to $Z$. Therefore
$$\dim\Lambda_{\underline{p}}=\max\left\{s_{\infty},
%\sup_{\mu\in\mathcal{M}(T)}\left\{\frac{h(\mu)}{\lambda(\mu)}:\mu(I_i)=p_i\forall i\in\N,\text{ }h(\mu)<\infty\right\}
\ \alpha_3(\underline{p})
\right\}$$
where
$$\alpha_3(\underline{p})=\sup_{\mu\in\mathcal{M}(T)}\left\{\frac{h(\mu)}{\lambda(\mu)}:\mu(I_i)=p_i \ \forall i\in\N,\text{ }h(\mu)<\infty\right\}.$$
We refer to these sets $\Lambda_{\underline{p}}$ as ``sets of digit frequency". This is because in the case where $T$ is the Gauss map, $T(x)=1/x\mod 1$, $A_n\phi_i(x)$ gives the frequency of $i$ in the first $n$ terms of the continued fraction expansion of $x$. In particular our work shows that the dimension of such a set is always bounded below by $s_{\infty}$ even if the frequencies sum to less than $1$. Note that $s_{\infty}=1/2$ when $T$ is the Gauss map. This problem has already been studied in the setting of continued fractions (\cite{FLM}), and in the countable state symbolic space (\cite{FLMW}). Our work shows that this phenomenon extends to more general countable branch expanding maps. We should also point out that there was a step missing from the proof in \cite{FLM} where the argument of how to go from the statement of Theorem \ref{main0} to Theorem \ref{main} was not given. The section on the proof of Theorem \ref{main} shows how this can be done.
\subsection{Harmonic averages for continued fractions}
For another example we again let $T$ be the Gauss map. If we just take one potential $\phi:[0,1]\backslash\Q\rightarrow\R$ defined by $\phi(x)=\frac{1}{a_1(x)}$ where $a_1(x)$ is the first digit in the continued expansion of $x$ then Theorem \ref{main} is still applicable. In particular if
for $\alpha\in [0,1]$, let
$$\Lambda_{\alpha}=\left\{x\in [0,1]\backslash\Q:\lim_{n\rightarrow\infty}\frac{\frac{1}{a_1(x)}+\frac{1}{a_2(x)}+\cdots+\frac{1}{a_n(x)}}{n}=\alpha\right\}$$
then we have
$$\dim\Lambda_{\alpha}=\max\left\{\frac{1}{2},\sup_{\mu\in\mathcal{M}(T)}\left\{\frac{h(\mu)}{\lambda(\mu)}:\int\phi\text{d}\mu=\alpha,\text{ }h(\mu)<\infty\right\}\right\}.$$
From this we can deduce that $$
\dim\Lambda_0=\Lambda_1=\frac{1}{2}.$$
    For $\Lambda_1$, note that the Dirac measure on the point $\frac{\sqrt{5}-1}{2}$ is the only $T$-invariant measure $\nu$ with $\int\phi\text{d}\nu=1$. However, despite the fact that this measure clearly has dimension $0$, the set $\Lambda_1$ still has dimension $\frac{1}{2}$.

Furthermore, in this case we can show that the only points where the dimension achieves the lower bound $\frac{1}{2}$ are the endpoints of the spectrum.
\begin{prop}
For all $\alpha\in (0,1)$
$\dim \Lambda_{\alpha}>\frac{1}{2}$.
\end{prop}
\begin{proof}
Fix $\alpha\in (0,1)$.
Consider the set  of irrationals $x$ for which the continued fraction expansion $a_1(x),a_2(x),\ldots$ satisfies that for some $N\in\N$ $a_i(x)>N$ for all $i\in \N$. We will denote this set $E_N$ and note that if we consider  the restriction of the Gauss map $T$ to
    the union of the intervals $I_j$  ($j\geq N$), then $E_N$ is its attractor and the corresponding value of $s_{\infty}$ is still $\frac{1}{2}$.
In \cite{JK} it is shown that $$\dim E_N\sim\frac{1}{2}+\frac{\log\log N}{2\log N}.$$
    Since $\frac{1}{2}<\dim E_N$,
    we can deduce that $E_N$ admits an ergodic measure of maximal dimension $\mu_N$ with $h(\mu_N)<\infty$. Note that for $N$ sufficiently large we have that $\lambda(\mu_N)\geq\log N$.

Take $\delta_1$ to be the Dirac measure at $\frac{\sqrt{5}-1}{2}$.  Then  $\delta_1$ is ergodic and $\int \phi\text{d}\delta_1=1$.
    Now  consider measures of the form
    $$
    \nu_p=p\mu_N+(1-p)\delta_1.
    $$
    If we choose  $p>\frac{\lambda(\delta_1)4\log N}{\log\log N\lambda(\mu_N)} $ then
we have
\begin{eqnarray*}
h(\nu_p)&=&ph(\mu_N)\geq p\left(\frac{1}{2}+\frac{\log\log N}{4\log N}\right)\lambda(\mu_N)\\
&=&\frac{p}{2}\lambda(\mu_N)+p\frac{\log\log N}{4\log N}\lambda(\mu_N)\\
&>&\frac{1}{2}(p\lambda(\mu_N)+(1-p)\lambda(\delta_1))=\frac{1}{2}\lambda(\nu_p).
\end{eqnarray*}
Thus $\frac{h(\nu_p)}{\lambda(\nu_p)}>\frac{1}{2}$. Furthermore since $\lim_{N\to\infty}\frac{\lambda(\delta(1))4\log N}{\log\log N\lambda(\mu_N)}=0$ and $\lim_{N\to\infty}\int\phi\text{d}\mu_N=0$,
we can choose $q$ such  that   $\frac{h(\nu_p)}{\lambda(\nu_p)}>\frac{1}{2}$ for all $p>q$ and $\alpha=\int \phi\text{d}\nu_p$ for some $p>q$.
\end{proof}

It is straightforward to adapt this argument to the case where $T$ is the Gauss map and  where $\phi$ is a bounded function with variations uniformly tending to $0$. This will show that the interior of the spectrum is strictly greater than $\frac{1}{2}$ . However this is not always the case for alternative choices of $T$. A simple counter-example  is when $P(-s_{\infty}\log |T'|)\leq 0$ and $\phi$ is any bounded potential. In this case $\dim \Lambda_{\alpha}=s_{\infty}$ for all $$\alpha\in \left[\inf_{\mu\in\mathcal{M}(T)}\left\{\int\phi\text{d}\mu\right\},\sup_{\mu\in\mathcal{M}(T)}\left\{\int\phi\text{d}\mu\right\}\right].$$

\subsection{Locally flat spectrum}
Here we look at single functions where the multifractal spectrum will have interesting phase transitions.
These are examples where the function $\alpha\to\dim\Lambda_{\alpha}$ has flat regions but for which the whole spectrum is not flat. Let $T$  be a piecewise linear map defined using a partition (similar maps are studied in \cite{KMS}) as follows. We consider a set of disjoint closed intervals $\{I_i\}_{i=1}^{\infty}$. Denote $s_{\infty}$ as before and let $$
K=\diam(I_1)^{s_{\infty}}\text{ and } C=\sum_{i=2}^{\infty}\diam(I_{i})^{s_{\infty}}.$$
We will assume that $C<1$, $K+C>1$. (These can be easily satisfied. For example, take $|I_n|\approx n^{-2}(\log n)^{-4}$.) Define $T$ to be the piecewise linear map which maps each interval $I_i$ bijectively to the interval $[0,1]$.   These conditions will ensure that $$
\dim\Lambda>s_{\infty},\qquad P(-s_{\infty}\log |T'|)<\infty.
    $$
    We will take $\phi=\chi_{I_1}$, that is the characteristic function for the interval $I_1$.
We will prove the following result.
\begin{thm}\label{flat-spectrum}
There exist $0<\alpha_*<\alpha^*<1$ such that $\dim\Lambda_{\alpha}=s_{\infty}$ for $\alpha\in [0,\alpha_*]\cup[\alpha^*,1]$ and $\dim\Lambda_{\alpha}>s_{\infty}$ for $\alpha\in (\alpha_*,\alpha^*)$.
\end{thm}

\subsection*{Proof of Theroem \ref{flat-spectrum}}
We will prove Theorem \ref{flat-spectrum} by a series of propositions and lemmas.
We start with the following general proposition.
\begin{prop}\label{negpres}Let $\phi:\Lambda\to\R$ have variations uniformly converging to $0$.
For any $\alpha \in \mathbb{R}$ if there exist $q,\delta$ such that
$$P(q(\phi-\alpha)-\delta\log |T'|)\leq 0,$$
then
$$\sup_{\mu\in\mathcal{M}(T)}\left\{\frac{h(\mu)}{\lambda(\mu)}:\int\phi\text{d}\mu=\alpha\text{ and }\lambda(\mu)<\infty\right\}\leq \delta.$$
\end{prop}
\begin{proof}
Let $\mu\in\mathcal{M}(T)$ such that $\int\phi\text{d}\mu=\alpha$ and $\lambda(\mu)<\infty$.  By the variational principle, we have
$$h(\mu)+\int (q(\phi-\alpha)-\delta\log |T'|)\text{d}\mu\leq 0.$$
So,
$$h(\mu)-\delta\lambda(\mu)\leq 0.$$
Thus ${h(\mu)}/{\lambda(\mu)}\leq\delta$ which completes the proof.
\end{proof}
Therefore, for our specific choice of $T$ and $\phi$
if we can find $q>0$ and $\alpha^{*}\in (0,1)$ such that $P(q(\phi-\alpha^{*})-s_{\infty}\log |T'|)=0$ then $\dim \Lambda_{\alpha}=s_{\infty}$ for all $\alpha\in (\alpha^{*},1)$. Similarly if we can find $q<0$ and $\alpha_*\in (0,1)$ such that $P(q(\phi-\alpha_*)-s_{\infty}\log |T'|)=0$ then $\dim \Lambda_{\alpha}=s_{\infty}$ for all $\alpha\in (0,\alpha_{*})$.

%In this setting
We are going to  show that we can indeed find such $\alpha_*,\alpha^*$. We can calculate
$$
P(q(\phi-\alpha)-s_{\infty}\log |T'|)=\log(Ke^{q}+C)-\alpha q.$$
By solving the equation $P(q(\phi-\alpha)-s_{\infty}\log |T'|)=0$, we have
$$\alpha(q)=\frac{\log(Ke^q+C)}{q} , \quad q\neq 0.$$
We then have the following lemma.
\begin{lem}
Such $\alpha_*,\alpha^*$ do exist.
\end{lem}
\begin{proof}
The function $\alpha(q)$ has the following properties:
\begin{enumerate}
\item[(1).]
The function $\alpha(q)$ is real analytic on both $(-\infty, 0)$ and $(0,\infty)$.
\item[(2).]
   \( \lim_{q\to \infty} \alpha(q)=1 \quad \text{and} \quad \lim_{q\to -\infty} \alpha(q)=0.\)
\item[(3).]
\( \lim_{q\to 0+} \alpha(q)=+\infty \quad \text{and} \quad \lim_{q\to 0-} \alpha(q)=-\infty.\)
\item[(4).]
   Under our conditions $K+C>1$ and $C<1$, $\alpha(q)<1$ for $q<0$ and $\alpha(q)>0$ for $q>0$ and the equation $\alpha(q)=0$ admits only one solution $$q=q_-= \log \frac{1-C}{K}<0$$ and the equation $\alpha(q)=1$ admits only one solution $$q=q_+= \log \frac{C}{1-K}>0.$$
\end{enumerate}

From the above properties, one can see the minimum and maximum of the following can be obtained:
\[
          \alpha^*= \inf_{q>0} \alpha(q) =\inf_{q>q_+} \alpha(q) \quad \text{and} \quad
          \alpha_*= \inf_{q<0} \alpha(q) =\inf_{q<q_-} \alpha(q).
\]
These are what we want.
\end{proof}
Thus we have that for any $\alpha\in [0,\alpha_*]\cup[\alpha^*,1]$ there exists $q$ such that
$P(q(\phi-\alpha)-s_{\infty}\log |T'|)\leq 0$ and so
by Proposition \ref{negpres} we have
$$\dim \Lambda_\alpha=s_{\infty}, \quad \forall \alpha\in [0, \alpha_*]\ \mbox{\rm and} \ \forall \alpha\in [\alpha^*, 1].$$

Now we need to show $$
\forall \alpha\in (\alpha_*, \alpha^*), \quad \dim_H \Lambda_\alpha>s_{\infty}.
$$

For $t\in [s_{\infty}, \dim \Lambda]$, denote
\[ K(t)=|I_1|^{t}, \quad \text{and} \quad C(t)=\sum_{i=2}^{\infty} |I_i|^{t}.\]

     Let $f(t,q)= P(q\phi-t\log |T'|)$. Then the dimension of the set $\Lambda_\alpha$ is the first component $t(\alpha)$ of the solution $(t(\alpha),q(\alpha))$ to the following system (see \cite{FLWW}):
\begin{eqnarray}\label{equations}
\left\{
     \begin{array}{ll}
        f(t,q)=q\alpha, \\
        \displaystyle \frac{\partial f}{\partial q}(t,q)= \alpha
     \end{array}
\right.
\end{eqnarray}
whenever such a solution exists.
By a simple calculation we have
\[ f(t,q)=\log(K(t)e^q+C(t)).\]
For a fixed $t$, let $f_t(q)=f(t,q)$.
\begin{lem}
For $\alpha\in (\alpha_*,\alpha^*)$ we have that $P(q(\phi-\alpha)-s_{\infty}\log|T'|)>0$ for all $q$ and that $P(q(\phi-\alpha)-(\dim\Lambda)\cdot\log |T'|)\leq 0$ for some $q\in\R$.
\end{lem}
\begin{proof}
The function $q \mapsto f_t(q)$ has the following properties:
\begin{enumerate}
\item[(1)] For $t\in (s_{\infty}, \dim\Lambda)$, the function $f_t(q)$ has two asymptotic lines $y= \log C(t)$ for $q\to -\infty$ and $y= x+ \log K(t)$ for $q\to \infty$. In particular note that for any $\alpha\in (0,1)$ there exists $q(\alpha,t)$ such that $f_t'(q(\alpha,t))=\alpha$.
\item[(2)]
    \[\alpha^*= \inf_{q>0} \frac{f_{s_{\infty}}(q)}{q}<1 \quad \text{and} \quad \alpha_*= \inf_{q<0} \frac{f_{s_{\infty}}(q)}{q}>0.  \]
\item[(3)] If $\alpha\in (\alpha_*,\alpha^*)$ then $f_{s_{\infty}}(q)=\alpha q$ has no solution.
\end{enumerate}
By property (3) and property (2) we can thus deduce that for $\alpha\in (\alpha_*,\alpha^*)$ and for any $q\in\R$
$$P(q(\phi-\alpha)-s_{\infty}\log|T'|)=f_{s_{\infty}}(q)-\alpha q>0,$$
which is the first part of the lemma.

By property (1) if we let $s=\dim\Lambda$ then there exists $q(\alpha,s)$ such that $f_{s}'(q(\alpha,s))=\alpha$. It then follows that there will be an equilibrium state $\mu_{q,s}$ such that $\int\phi\text{d}\mu_{q,s}=\alpha$ and
$$f_{s}(q(\alpha,s))=\alpha q-s\lambda(\mu_{q,s})+h(\mu_{q,s})\leq \alpha q.$$
Thus the second part of the lemma follows.
\end{proof}
Due to the fact that $f(t,q)$ depends analytically on $t,q$ in the region $t>s_{\infty},q\in\R$, we can now assert that
   for $\alpha\in(\alpha_*, \alpha^*)$ there exists $t(\alpha)\in (s_{\infty},\dim\Lambda)$ which is the first coordinate  of the solution $(t(\alpha),q(\alpha))$ to (\ref{equations}) and thus $\dim \Lambda_{\alpha}= t(\alpha)$. This completes the proof of Theorem \ref{flat-spectrum}.

We can also deduce that if $\mu_{SRB}$ is the equilibrium state for the potential $-(\dim\Lambda)\cdot\log |T'|$ and $\tilde{\alpha}=\int\phi\text{d}\mu_{SRB}$ then the function $\alpha\to\dim \Lambda_{\alpha}$ is strictly increasing on $(\alpha_*,\tilde{\alpha})$ and strictly decreasing on $(\tilde{\alpha},\alpha_*)$ and by the implicit function theorem varies analytically in the region $(\alpha_*,\alpha^*)$.


\begin{thebibliography}{FLWW09}

\bibitem[Bes35]{bes}
A.~S. Besicovitch, \emph{On the sum of digits of real numbers represented in
  the dyadic system}, Math. Ann. \textbf{110} (1935), no.~1, 321--330.
  \MR{1512941}

\bibitem[BS00]{BSc}
L.~Barreira and J.~Schmeling, \emph{Sets of ``non-typical'' points have full
  topological entropy and full {H}ausdorff dimension}, Israel J. Math.
  \textbf{116} (2000), 29--70. \MR{1759398 (2002d:37040)}

\bibitem[BS01]{BS}
L.~Barreira and B.~Saussol, \emph{Variational principles and mixed multifractal
  spectra}, Trans. Amer. Math. Soc. \textbf{353} (2001), no.~10, 3919--3944
  (electronic). \MR{1837214 (2002d:37048)}

\bibitem[BSS02a]{BSS2}
L.~Barreira, B.~Saussol, and J.~Schmeling, \emph{Distribution of frequencies of
  digits via multifractal analysis}, J. Number Theory \textbf{97} (2002),
  no.~2, 410--438. \MR{1942968 (2003m:11124)}

\bibitem[BSS02b]{BSS}
\bysame, \emph{Higher-dimensional multifractal analysis}, J. Math. Pures Appl.
  (9) \textbf{81} (2002), no.~1, 67--91. \MR{1994883 (2004g:37038)}

\bibitem[Caj81]{Caj}
H.~Cajar, \emph{Billingsley dimension in probability spaces}, Lecture Notes in
  Mathematics, vol. 892, Springer-Verlag, Berlin, 1981. \MR{654147 (84a:10055)}

\bibitem[Dur97]{Du}
A.~Durner, \emph{Distribution measures and {H}ausdorff dimensions}, Forum Math.
  \textbf{9} (1997), no.~6, 687--705. \MR{1480551 (98i:11060)}

\bibitem[Egg49]{E}
H.~G. Eggleston, \emph{The fractional dimension of a set defined by decimal
  properties}, Quart. J. Math., Oxford Ser. \textbf{20} (1949), 31--36.
  \MR{0031026 (11,88e)}

\bibitem[FF00]{FF}
A.-H. Fan and D.-J. Feng, \emph{On the distribution of long-term time averages
  on symbolic space}, J. Statist. Phys. \textbf{99} (2000), no.~3-4, 813--856.
  \MR{1766907 (2002d:82003)}

\bibitem[FFW01]{FFW}
A.-H. Fan, D.-J. Feng, and J.~Wu, \emph{Recurrence, dimension and entropy}, J.
  London Math. Soc. (2) \textbf{64} (2001), no.~1, 229--244. \MR{1840781
  (2003b:37028)}

\bibitem[FLM10]{FLM}
A.-H. Fan, L.-M. Liao, and J.-H. Ma, \emph{On the frequency of partial
  quotients of regular continued fractions}, Math. Proc. Cambridge Philos. Soc.
  \textbf{148} (2010), no.~1, 179--192. \MR{2575381 (2011a:11146)}

\bibitem[FLMW10]{FLMW}
A.-H. Fan, L.-M. Liao, J.-H. Ma, and B.-W. Wang, \emph{Dimension of
  {B}esicovitch-{E}ggleston sets in countable symbolic space}, Nonlinearity
  \textbf{23} (2010), no.~5, 1185--1197. \MR{2630097 (2011h:28007)}

\bibitem[FLP08]{FLP}
A.-H. Fan, L.-M. Liao, and J.~Peyri{\`e}re, \emph{Generic points in systems of
  specification and {B}anach valued {B}irkhoff ergodic average}, Discrete
  Contin. Dyn. Syst. \textbf{21} (2008), no.~4, 1103--1128. \MR{2399452
  (2009b:37029)}

\bibitem[FLW02]{FLW}
D.-J. Feng, K.-S. Lau, and J.~Wu, \emph{Ergodic limits on the conformal
  repellers}, Adv. Math. \textbf{169} (2002), no.~1, 58--91. \MR{1916371
  (2003j:37036)}

\bibitem[FLWW09]{FLWW}
A.-H. Fan, L.-M. Liao, B.-W. Wang, and J.~Wu, \emph{On {K}hintchine exponents
  and {L}yapunov exponents of continued fractions}, Ergodic Theory Dynam.
  Systems \textbf{29} (2009), no.~1, 73--109. \MR{2470627 (2009m:11124)}

\bibitem[GR09]{GR}
K.~Gelfert and M.~Rams, \emph{The {L}yapunov spectrum of some parabolic
  systems}, Ergodic Theory Dynam. Systems \textbf{29} (2009), no.~3, 919--940.
  \MR{2505322 (2010f:37064)}

\bibitem[Hof10]{H}
F.~Hofbauer, \emph{Multifractal spectra of {B}irkhoff averages for a piecewise
  monotone interval map}, Fund. Math. \textbf{208} (2010), no.~2, 95--121.
  \MR{2640068 (2011e:37072)}

\bibitem[IJ10]{IJ}
G~Iommi and T.~Jordan, \emph{Multifractal analysis of birkhoff averages for
  countable markov maps}, Preprint (2010).

\bibitem[JK10]{JK}
J.~Jaerisch and M.~Kesseb{\"o}hmer, \emph{The arithmetic-geometric scaling
  spectrum for continued fractions}, Ark. Mat. \textbf{48} (2010), no.~2,
  335--360. \MR{2672614 (2011g:11144)}

\bibitem[JMU05]{JMU}
O.~Jenkinson, R.~D. Mauldin, and M.~Urba{\'n}ski, \emph{Zero temperature limits
  of {G}ibbs-equilibrium states for countable alphabet subshifts of finite
  type}, J. Stat. Phys. \textbf{119} (2005), no.~3-4, 765--776. \MR{2151222
  (2006g:37051)}

\bibitem[KMS12]{KMS}
M.~Kesseb\"{o}hmer, S.~Munday, and B.~Stratmann, \emph{Strong renewal theorems
  and lyapunov spectra for $\alpha$-farey and $\alpha$-l\"{u}roth systems},
  Ergodic Theory Dynam. Systems \textbf{32} (2012), no.~3, 989--1017.

\bibitem[KR12]{KR}
A~K\"aenm\"aki and H.~Reeve, \emph{Multifractal analysis of birkhoff averages
  for typical infinitely generated self-affine sets}, Preprint (2012).

\bibitem[MU03]{MUbook}
R.~D. Mauldin and M.~Urba{\'n}ski, \emph{Graph directed {M}arkov systems},
  Cambridge Tracts in Mathematics, vol. 148, Cambridge University Press,
  Cambridge, 2003, Geometry and dynamics of limit sets. \MR{2003772
  (2006e:37036)}

\bibitem[Oli98]{Oli98}
E.~Olivier, \emph{Analyse multifractale de fonctions continues}, C. R. Acad.
  Sci. Paris S\'er. I Math. \textbf{326} (1998), no.~10, 1171--1174.
  \MR{1650242 (99h:58109)}

\bibitem[Oli99]{Oli1}
\bysame, \emph{Dimension de {B}illingsley d'ensembles satur\'es}, C. R. Acad.
  Sci. Paris S\'er. I Math. \textbf{328} (1999), no.~1, 13--16. \MR{1674409
  (2000b:28021)}

\bibitem[Oli00]{Oli00}
\bysame, \emph{Structure multifractale d'une dynamique non expansive d\'efinie
  sur un ensemble de {C}antor}, C. R. Acad. Sci. Paris S\'er. I Math.
  \textbf{331} (2000), no.~8, 605--610. \MR{1799097 (2002g:37033)}

\bibitem[Ols02]{Ol02}
L.~Olsen, \emph{Divergence points of deformed empirical measures}, Math. Res.
  Lett. \textbf{9} (2002), no.~5-6, 701--713. \MR{1906072 (2003k:37038)}

\bibitem[Ols03a]{O}
\bysame, \emph{Multifractal analysis of divergence points of deformed measure
  theoretical {B}irkhoff averages}, J. Math. Pures Appl. (9) \textbf{82}
  (2003), no.~12, 1591--1649. \MR{2025314 (2004k:37036)}

\bibitem[Ols03b]{Ol03}
\bysame, \emph{Small sets of divergence points are dimensionless}, Monatsh.
  Math. \textbf{140} (2003), no.~4, 335--350. \MR{2026104 (2005a:37034)}

\bibitem[Ols04]{O2}
\bysame, \emph{Applications of multifractal divergence points to sets of
  numbers defined by their {$N$}-adic expansion}, Math. Proc. Cambridge Philos.
  Soc. \textbf{136} (2004), no.~1, 139--165. \MR{2034019 (2004j:11090)}

\bibitem[OW03]{OW03}
L.~Olsen and S.~Winter, \emph{Normal and non-normal points of self-similar sets
  and divergence points of self-similar measures}, J. London Math. Soc. (2)
  \textbf{67} (2003), no.~1, 103--122. \MR{1942414 (2003i:28009)}

\bibitem[OW07]{OW}
\bysame, \emph{Multifractal analysis of divergence points of deformed measure
  theoretical {B}irkhoff averages. {II}. {N}on-linearity, divergence points and
  {B}anach space valued spectra}, Bull. Sci. Math. \textbf{131} (2007), no.~6,
  518--558. \MR{2351308 (2010b:28023)}

\bibitem[PS07]{PS07}
C.-E. Pfister and W.~G. Sullivan, \emph{On the topological entropy of saturated
  sets}, Ergodic Theory Dynam. Systems \textbf{27} (2007), no.~3, 929--956.
  \MR{2322186 (2008f:37036)}

\bibitem[PU10]{PU}
F.~Przytycki and M.~Urba{\'n}ski, \emph{Conformal fractals: ergodic theory
  methods}, London Mathematical Society Lecture Note Series, vol. 371,
  Cambridge University Press, Cambridge, 2010. \MR{2656475 (2011g:37002)}

\bibitem[PW01]{PW}
Y.~Pesin and H.~Weiss, \emph{The multifractal analysis of {B}irkhoff averages
  and large deviations}, Global analysis of dynamical systems, Inst. Phys.,
  Bristol, 2001, pp.~419--431. \MR{1858487 (2002m:37034)}

\bibitem[Ree11]{R}
H.~Reeve, \emph{Infinite non-conformal iterated function systems}, Israel J.
  Math. {to appear, arxiv.org/abs/1105.3137} (2011).

\bibitem[Tem01]{T}
A.~A. Tempelman, \emph{Multifractal analysis of ergodic averages: a
  generalization of {E}ggleston's theorem}, J. Dynam. Control Systems
  \textbf{7} (2001), no.~4, 535--551. \MR{1854035 (2002g:37008)}

\bibitem[TV03]{TV}
F.~Takens and E.~Verbitskiy, \emph{On the variational principle for the
  topological entropy of certain non-compact sets}, Ergodic Theory Dynam.
  Systems \textbf{23} (2003), no.~1, 317--348. \MR{1971209 (2004b:37041)}

\bibitem[Vol58]{Vol58}
B.~Volkmann, \emph{\"{U}ber {H}ausdorffsche {D}imensionen von {M}engen, die
  durch {Z}ifferneigenschaften charakterisiert sind. {VI}}, Math. Z.
  \textbf{68} (1958), 439--449. \MR{0100578 (20 \#7008)}

\end{thebibliography}
\end{document}